\documentclass[11pt]{amsart}
\usepackage[latin1]{inputenc}
\usepackage{amsmath,amsthm,amsopn,amstext,amscd,amsfonts,amssymb}
\usepackage{amssymb}
\usepackage{amsfonts}
\usepackage{lscape}
\long
\def\salta#1{\relax}
\newcommand{\R}{{I\!\!R}}
\newcommand{\N}{{I\!\!N}}

\newcommand{\car}{{\raise2pt\hbox{$\chi$}}}

\newcommand{\dob}{{\mathcal D}^{1,2}}

\renewcommand{\t }{\tau }

\newtheorem{Theorem}{Theorem}[section]
\newtheorem{Corollary}[Theorem]{Corollary}

\newtheorem{Lemma}[Theorem]{Lemma}

\newcommand{\B}{\mathcal{B}}
\newcommand{\C}{\mathcal{C}}

 \newcommand{\NN}{\mathcal{N}}

\newcommand{\F}{\mathcal{F}}
 \newcommand{\W}{\mathcal{W}}

\begin{document}
\title[]{Blow-up analysis for a Hardy-Sobolev equation on compact Riemannian manifolds with application to the existence of  solutions.}
\author[Y. Maliki, F.Z. Terki]{Y. Maliki$^*$ and F.Z. Terki}
\address{Y. Maliki, F.Z. Terki \hfill \break\indent D\'epartement de
Math\'ematiques, Universit\'e Abou Bakr Belkaïd, Tlemcen, \hfill\break%
\indent Tlemcen 13000, Algeria.} \email{\texttt{malyouc@yahoo.fr,
fatimazohra113@yahoo.fr}}

\date{}
\maketitle

\begin{abstract} On a compact Riemannian manifold, we study a singular elliptic equation with  critical Sobolev exponent and critical  Hardy potential. In a first part, we prove an  $H^2_1$ type decomposition result for Palais-Smale sequences of the associated energy functional. In a second part, we apply the decomposition result to obtain solutions of different energy levels.
\end{abstract}
\section{Introduction}
Let $(M,g$) be a compact $(n\geq 3)-$dimensional Riemannian manifold. Denote by $\delta_g>0$ its injectivity radius. For a fixed point $ p\in M$, define (after \cite{Madani}) on $M$ a function $\rho_p$  as follows
\begin{equation*}
    \rho_p(x)=\left\{
       \begin{array}{ll}
         dist_g(p,x), &x\in B(p,\delta_g),\\
         \delta_g, &x\in M\setminus B(p,\delta_g)
       \end{array}
     \right.
\end{equation*}
Let $h$ be a continuous functions on $M\setminus\{p\}$ and consider the following Hardy-Sobolev equation:
\begin{equation} \label{0.1}
\Delta _{g}u-\frac{h}{\rho_p ^{2}}u=|u|^{2^*-2}u,
\tag{$E$}
\end{equation}
where $\Delta_g=-\text{div}(\nabla_g u)$ is the Laplacian operator on the manifold $(M,g)$ and $2^*=\frac{2n}{n-2}$ is the Sobolev critical exponent.\\
Equation \eqref{0.1}, when the Hardy potential $\frac{h}{\rho_p ^{2}}u $ replaced by $\frac{n-2}{4(n-1)}Scal_g$, is the famous Yamabe  equation arising from the conformal deformation of the metric $g$ and which has been largely studied( see \cite{Aubin} for an exposure of the main pioneering works). When  the function $\rho_p$ is of power $0<\gamma<2$, equation\eqref{0.1} appears as a case of equations that arise in the study of conformal deformation to constant scalar curvature of metrics which are smooth only in some ball $B_p(\delta)$; it is a kind of a singular  Yamabe problem that has been formulated and studied in \cite{Madani}.\\
On the Euclidean space $\R^n$, equation \eqref{0.1}, with a function $K$ involved in the right-hand side, has been studied in \cite{D. Smet}. The author obtained some existence results after having proved a result on decomposition of Palais-Smale sequences of the functional energy. About this decomposition result, the author showed that the singular term does interfere in the decomposition and gives rise to a second type of bubbles in addition, of course, to  bubbles which results from the existence of the Sobolev exponent.\\
 In this paper, we aim at extending this decomposition result to the context of compact Riemannian manifolds and equations like \eqref{0.1}. To achieve this aim, we follow the authors in \cite{Druet-hebbey-robert} in their constructions when they extended, to Yamabe type equations on compact Riemannian manifold, the Struwe's \cite{Struwe} decomposition result. More precisely, for our decomposition result we prove that Palais-Smale sequences split into the sum of a solution of equation\eqref{0.1} and bubbles which construct from solutions $u$ and $v$ on $ D^{1,2}(\R^n)$ of equations
 \begin{equation}\label{0.2}
    \Delta_{_{\R^n}} u =|u|^{\frac{4}{n-2}}u,
\end{equation}
and
\begin{equation}\label{0.3}
\Delta_{_{\R^n}}v-h(p)\frac{v}{|x|^2}=|v|^{\frac{4}{n-2}}v.
\end{equation}
As an application, we use the decomposition result to determine energy regions in which  Palais-Smale sequences are compact  and then converge, up to subsequences, to solutions of \eqref{0.1} of different energy levels.
\section{Notations and background materials}
In the following, we introduce some notations and materials that will be used throughout the paper.\\
 Denote by $D^{1.2}(\R^n)$, the Euclidean Sobolev space defined as the completion, with respect to the norm
\begin{equation*}
    ||u||_{D^{1.2}(\R^n)}^2=\int_{\R^n}|\nabla u|^2dx,
\end{equation*}
of the space $\C_o^\infty(\R^n)$ of smooth functions on $\R^n$ with compact support.\\
Let $K(n,2)$ denote the best constant in the Sobolev inequality
\begin{equation*}
    \left(\int_{\R^n}|u^{2^*}|dx\right)^{\frac{2}{2^*}}\le K^2(n,2)\int_{\R^n}|\nabla u|^2dx.
\end{equation*}
It is well known that the exact value of $K(n,2)$ is
\begin{equation*}
    K(n,2)=\sqrt{\frac{4}{n(n-2)w_n^{\frac{n}{2}}}},
\end{equation*}
where $w_n$ denotes the volume of the unit sphere in the Euclidean space $\R^{n+1}$.\\
Let $K(n,2,-2)=\frac{2}{(n-2)}$ denote  the best constant in the Hardy inequality
on $D^{1.2}(\R^n)$,
\begin{equation*}
 \int_{\R^n}\frac{u^2}{|x|^2}dx\le K^2(n,2,-2)\int_{\R^n}|\nabla u|^2dx.
\end{equation*}
Let  $0<\lambda< \frac{1}{K(n,2,-2)^{2}}$ and consider on $D^{1.2}(\R^n)$ the equation
\begin{equation}\label{2.2}
\Delta_{_{\R^n}}u-\lambda\frac{u}{|x|^2}=|u|^{\frac{4}{n-2}}u.
\end{equation}
By a classification result in \cite{Terracini}, positive solutions of \eqref{2.2} are the family of functions
\begin{equation}\label{1.2'}
    U_\mu(x)=\mu^{\frac{2-n}{2}}U\left(\frac{x}{\mu}\right),\mu>0
\end{equation}
where
\begin{equation*}
U(x)=(n(n-2))^{\frac{n-2}{4}}\left(\frac{a\left\vert x\right\vert ^{a-1}}{%
1+\left\vert x\right\vert ^{2a}}\right)^{\frac{n}{2}-1}, x\in\R^n
\end{equation*}
 and
\begin{equation}\label{1.2}
a=\sqrt{1-\lambda K(n,2,-2)^2}.
\end{equation}
Moreover, the family of functions $U_\lambda(x) $ satisfies
\begin{eqnarray}\label{0.4}
 \nonumber \inf_{u\in D^{1,2}(\R^n)\setminus\{0\}}\frac{\int_{\R^n}\left(|\nabla u|^2-\lambda\frac{u^2}{|x|^2}\right)dx}{(\int_{\R^n}|u|^{2^*}dx)^{\frac{2}{2^*}}}&=&\frac{\int_{\R^n}\left(|\nabla U_\lambda|^2-\lambda\frac{U_\lambda^2}{|x|^2}\right)dx}{(\int_{\R^n}|U_\lambda|^{2^*}dx)^{\frac{2}{2^*}}}
\\&=&\frac{(1-\lambda K(n,2,-2)^2)^{\frac{n-1}{n}}}{K(n,2)^2},
\end{eqnarray}
On the compact Riemannian manifold $M$, we consider  the Sobolev space $H_{1}^{2}(M)$ consisting of
the completion of $\C^{\infty}(M)$ with respect to the norm
$$||u||^2_{H_{1}^{2}(M)}=\int_M(|\nabla u|^2+u^2)dv_g.$$
By the Rellich-Kondrakov theorem ( see \cite{Hebey} ), if $M$ is compact and $q<2^*=\frac{2n}{n-2}$, the inclusion $H_{1}^{2}(M) \subset L_q(M)$ is compact. If $q=2^*$, the inclusion is only continuous. On the Sobolev space $H_{1}^{2}(M)$, the following optimal Sobolev inequality holds ( see \cite{Hebey2} , Theorem 4.6)). For any $u\in H_{1}^{2}(M)$, there exists a positive constant $B$ such that
\begin{equation}\label{1.1}
 ||u||^2_{L_{2^*}(M)}\le K^2(n,2)||\nabla u||^2_{L_2(M)}+B||u||^2_{L_2(M)}.
\end{equation}
We denote by $L_2(M,\rho_p^2)$ the space of functions on $M$ such that $\int_M\frac{u^2}{\rho_p^2}<\infty$. This space is endowed with the norm
\begin{equation*}
    ||u||_{L_2(M,\rho_p^2)}^2=\int_M\frac{u^2}{\rho_p^2}dv_g.
\end{equation*}
In \cite{Madani}, it is shown that the Sobolev space $H^2_1(M)$is continuously embedded in $L_2(M,\rho_p^2)$ and the following Hardy inequality  on $H^2_1(M)$ holds: for every $\varepsilon>0$ there exists a
positive constant $A(\varepsilon)$ such that for any $u\in
H^2_1(M)$,
\begin{equation}\label{2.1}
\int_M \frac{u^2}{\rho^2_p}dv_g\le(
K^2(n,2,-2)+\varepsilon)\int_M|\nabla
u|^2dv_g+A(\varepsilon)\int_Mu^2dv_g,
\end{equation}
If $u$ is supported in a ball $B(p,\delta),0<\delta<\delta_g$,  then
\begin{equation*}
\int_{B(p,\delta)} \frac{u^2}{\rho^2_p}dv_g\le
K_\delta(n,2,-2)\int_{B(p,\delta)}|\nabla u|^2dv_g,
\end{equation*}
with $ K_\delta(n,2,-2)$ goes to $K(n,2,-2)$ when $\delta$ goes to
$0$. \\
In the paper, we will denote by $B(a,r)$ a ball of center
$a$ and radius $r>0$, the point $a$ will be specified either in $M$
or in $\R^n$, and $B(r)$ is a ball in $\R^n$ of center $0$ and
radius $r>0$.\\
Finally, we denote by $\eta_\delta$, where $\delta>0$, a cut-off function that satisfies $\eta_\delta(x)=1, x\in B(a,\delta), 0\le\eta_\delta\le1, x\in B(a,2\delta) $, $\eta_\delta(x)=0, x\in\R^n\setminus B(a,\delta)$ and $|\nabla\eta_\delta|\le C, x\in B(a,2\delta)$.\\
\section{Decomposition of Palais-Smale sequences}
Let $J_h$ be the functional defined on $H_1^2(M)$ by
\begin{equation*}
J_h(u)=\frac{1}{2}\int_M(|\nabla
u|^2-\frac{h}{\rho^2}
u^2)dv_g-\frac{1}{2^*}\int_M|u|^{2^*}dv_g.
\end{equation*}
A Palais-Smale sequence $u_m$ of $J_h $ at a level
$\beta$ is defined to be the sequence that satisfies $J_h(u_m)\to
\beta $ and $DJ_h(u_m)\varphi\to 0,\forall\varphi \in H^2_1(M)$.\\
In this section,  we prove  an $H^2_1-$ type decomposition theorem for Palais-Smale sequences for the functional $J_h$. We follow closely  the blow-up theory given in \cite{Druet-hebbey-robert} where the authors establish a decomposition result for a regular elliptic equation on compact manifolds and prove that a sequence of solutions of this equation decomposes into the sum of a solution $u_o$  of a limiting equation and bubbles which are solutions of equation\eqref{0.2}. The energy of this sequence decomposes, in the same manner,  into the sum of the energy of $u_o$ and the energy of bubbles. This result is known as the $H_1^2-$Decomposition for Palais-Smale sequences. In our case, the singular term interferes in the decomposition process and appeals to an analysis near the singular point $p$ to be done. Inspired by a decomposition result  in \cite{D. Smet}, we show that two kinds of bubbles contribute in the decomposition of Palais-Smale sequences. Note that in \cite{Maliki}, we proved, by the same techniques, a decomposition result for an arbitrarily bounded energy sequence of solution of equation \eqref{0.1}.\\
Before we formulate our decomposition theorem, we introduce on $D^{1,2}(\R^n)$ the   functionals
\begin{eqnarray*}
J(u)&=& \frac{1}{2}\int_{\R^n}|\nabla
u|^2dx-\frac{1}{2^*}\int_{\R^n}|u|^{2^*}dx, \text{ and } \\
 J_\infty (u)&=& \frac{1}{2}\int_{\R^n}|\nabla
u|^2dx- \frac{h(p)}{2}\int_{\R^n}\frac{u^2}{|x|^2}
dx-\frac{1}{2^*}\int_{\R^n}|u|^{2^*}dx.
\end{eqnarray*}
Now, we state the following decomposition  theorem:
\begin{Theorem}\label{thm3.6} Let $(M,g)$ be a compact Riemannian manifold with
$dim(M)=n\ge3$ and let $h$ be a continuous function on $M$ that  on the point $p\in M$, it satisfies $0<h(p)<\frac{1}{K(n,2,-2)^2}$.\\
Let $u_m$ be a Plais-Smale sequence of the functional $J_h$ at level $\beta$. Then, there
exist $k \in \N$, sequences $ R_m^i>0,R_m^i\underset{m\to\infty}{\to}0$, $\ell\in\N^n$
sequences $\t_m^j>0,\t_m^j\underset{m\to\infty}{\to}0$,
converging sequences  $x_m^j\to x_o^j\neq p$ in $M$, a solution
$u\in H^2_1(M)$ of \eqref{0.1},  solutions $v_i\in D^{1,2}(\R^n)$
 of \eqref{0.2} and nontrivial solutions $\nu_j\in
D^{1,2}(\R^n)$ of \eqref{0.3} such that up to a subsequence
\begin{eqnarray*}
u_m&=&u+\sum_{i=1}^{k}(R^i_m)^{\frac{2-n}{n}}
\eta_r(\exp^{-1}_p(x))v_i((R_m^i)^{-1}\exp^{-1}_p(x))\\&+&\sum_{j=1}^{\ell}(\tau^i_m)^{\frac{2-n}{n}}
\eta_r(\exp^{-1}_{x_m^j}(x))\nu_j((\tau_m^j)^{-1}\exp^{-1}_{x_m^j}(x))+\W_m,\\&&
\text{ with } \W_m\to 0 \text{ in }H^1_2(M),
\end{eqnarray*}
and
\begin{equation*}
J_h(u_m)=J_h(u)+\sum_{i=1}^k
J_\infty(v_i)+\sum_{j=1}^l J(\nu_j)+o(1).
\end{equation*}
\end{Theorem}
The proof of the above theorem goes through several steps that we organize under the form of lemmas
\begin{Lemma}\label{lem2.1} Let $u_m$ be a Palais-Smale sequence for
$J_h$ at level $\beta$  that converges to a function $u$ weakly
in $H^2_1(M)$ and $L_2(M,\rho_p^2)$, strongly in $L_q(M), 1\le
q<2^*$ and almost everywhere in $M$. Then, $u$ is a weak solution of \eqref{0.1} and  the sequence
$v_m=u_m-u$ is a sequence of Palais-Smale for $J_h$ such that
$$J_h(v_m)=\beta-J_h(u)+o(1).$$
\end{Lemma}
\begin{proof}
Let $u_m$ be a Palais-Smale sequence for $J_h$, at level $\beta$. Then,   $DJ_h(u_m)u_m=o(||u_m||_{H^2_1(M)})$  which implies that
\begin{equation*}
J_h(u_m)=\frac{1}{n}\int_Mf|u_m|^{2^*}dv_g=\beta+o(1)+o(||u_m||_{H^2_1(M)}).
\end{equation*}
 which means that $u_m$ is bounded in
$L_{2^*}(M)$ and then in $L_2(M)$. Furthermore, we have
\begin{equation*}
\int_M|\nabla u_m|^2dv_g=
nJ_h(u_m)+\int_M\frac{h}{\rho_p^2}u_m^2dv_g+o(||u_m||_{H^2_1(M)})
\end{equation*}
By continuity of $h$ on $p$, for all
$\epsilon>0$ there exists $\delta>0$ such that
\begin{eqnarray*}
&\int_M|\nabla u_m|^2dv_g\le
n\beta+(\varepsilon+h(p))\int_{B(p,\delta)}\frac{u_m^2}{\rho_p^2}dv_g
&\\&+\delta^{-2}\int_{M\setminus B(p,\delta)}h_\alpha
u_m^2dv_g+o(||u_m||_{H^2_1(M)}),&
\end{eqnarray*}
then, by applying Hardy inequality \eqref{2.1} that for every
$\varepsilon>0$ small there exists a constant $A(\varepsilon)$ such
that
\begin{eqnarray*}
&\int_M|\nabla u_m|^2dv_g\le
n\beta+(\varepsilon+h_\alpha(p))(\varepsilon+
K^2(n,2,-2))\int_{M}|\nabla u_m|^2dv_g
&\\&+A(\varepsilon)\int_{M} u_m^2dv_g+o(||u||_{H^2_1(M)})+o(1)&
\end{eqnarray*}
since $0<h_\alpha(p)<\frac{1}{K^2(n,2,-2)}$, we can find
$\varepsilon >0$ small such that
$1-(\varepsilon+h_\alpha(p))(\varepsilon+ K^2(n,2,-2))>0$ which
implies that $\int_M|\nabla u_m|^2dv_g$ is bounded.\\
Now, if the sequence $u_m$ converges to a function $u$ weakly in $H^2_1(M)$ and $L_2(M,\rho_p^2)$, strongly in $L_2(M)$ and almost everywhere in $M$, then $u$ must satisfy
\begin{equation}\label{3.5}
DJ_h(u)\varphi=0,\forall\varphi \in H^2_1(M).
\end{equation}
 In fact, the sequence $|u_m|^{2^*-2}u_m$ is
bounded in $L_{\frac{2^*}{2^*-1}}(M)$ and converges almost
everywhere to $|u|^{2^*-2}u$ , we get that $u_m^{2^*-2}u_m$
converges weakly in $L_{\frac{2^*}{2^*-1}}(M)$ to $u^{2^*-2}u$.
This clearly implies that \eqref{3.5} is satisfied.\\
Moreover, for $\varphi\in
 H_1^2(M)$, we can write
\begin{equation*}
DJ_h(v_m)\varphi=DJ_h(u_m)\varphi-DJ_h(u)\varphi+
\Phi(v_m)\varphi+o(1),
\end{equation*}
 with
  \begin{eqnarray*}
\Phi(v_m)\varphi&=&\int_M\left(|v_m+u|^{2^*-2}(v_m+u)-
|v_m|^{2^*-2}v_m-|u|^{2^*-2}u\right)\varphi dv_g.
\end{eqnarray*}
Knowing that there exists a positive constant $C$ independent of $m$ such that
\begin{equation*}
\mid|v_m+u|^{2^*-2}(v_m+u)-
|v_m|^{2^*-2}v_m-|u|^{2^*-2}u\mid\le
C(|v_m|^{2^*-2}|u|+|u|^{2^*-2}|v_m|),
\end{equation*}
 we get, after applying Hölder inequality, that there exists a  positive constant $C$ such that
\begin{equation*}
|\Phi(v_m)\varphi|\le C
\left(\||v_m|^{2^*-2}|u|\|_{L_{\frac{2^*}{2^*-1}}(M)}+
\||u|^{2^*-2}|v_m|\|_{L_{\frac{2^*}{2^*-1}}(M)}\right)\|\varphi\|_{L_{2^*}(M)},
\end{equation*}
which gives that $\Phi(v_m)\varphi=o(1),\forall\varphi \in H^2_1(M)$, since both
$\frac{2^*(2^*-2)}{2^*-1}$ and $ \frac{2^*}{2^*-1}$ are smaller than
$2^*$ and the inclusion of $H^2_1(M)$ in $L_q(M)$ is compact for $q<2^*$. By \eqref{3.5}, we get then that
\begin{equation}\label{}
  DJ_h(v_m)\varphi=o(1).
\end{equation}
On the other hand, by the weakly converges in $H^2_1(M)$ and $L_2(M,\rho_p^2)$, we can also write
\begin{equation*}
 J_h(v_m)=J_h(u_m)-J_h(u)+\Psi(u_m)+o(1),
 \end{equation*}
 with
 $$\Psi(u_m)=\frac{1}{2^*}\int_M(|u_m|^{2^*}-|u|^{2^*}-|v_m|^{2^*})dv_g,$$
 which by the Brezis-Lieb convergence Lemma equals to $o(1)$, hence we obtain
\begin{equation*}
J_h(v_m)=\beta-J_h(u)+o(1).
\end{equation*}
This ends the proof of the lemma.
\end{proof}
\begin{Lemma}\label{lem3.2} Let $v_m$ be a Palais-Smale sequence of $J_h$ at level $\beta$ that converges
weakly to $0$ in $H_1^2(M)$. If
$\beta<\beta^*=\frac{\left(1-h(p)K^2(n,2,-2)\right)^{\frac{n}{2}}}{nK(n,2)^n}$, then  $v_m$ converges strongly to $0$ in $H_1^2(M)$.
\end{Lemma}
\begin{proof} Let $v_m$ is a Palais-Smale sequence of $J_h$ at level
$\beta$ that converges to $0$ weakly in $H_1^2(M)$, then
$\int_Mv_m^2dv_g=o(1)$ and
\begin{equation*}
\beta=\frac{1}{n}\int_M(|\nabla
v_m|^2-\frac{h}{\rho^2_p}v_m^2)dv_g=\frac{1}{n}\int_M|v_m|^{2^*}dv_g+o(1).
\end{equation*}
This implies that $\beta\ge0$. Hence, on the one hand, by Hardy
inequality \eqref{2.1} we get as in Lemma 3.2, that for small enough
$\varepsilon>0$,
\begin{equation}\label{eqn3.3}
\int_M|\nabla
v_m|^2dv_g\le\frac{n\beta}{1-[(h(p)+\varepsilon)(\varepsilon+
K^2(n,2,-2))]}+o(1),
\end{equation}
and on the other hand, by Sobolev inequality \eqref{1.1}, we also
get
\begin{equation}\label{eqn3.4}
\int_M|\nabla v_m|^2dv_g \ge \left(\frac{n\beta}{K^{2^*}(n,2)}\right)^{\frac{2}{2^*}}+o(1).
\end{equation}
Now, suppose that $\beta>0$, then the above inequalities
\eqref{eqn3.3} and \eqref{eqn3.4} , for  $m$ big  enough, give
\begin{equation*}
\beta\ge
\frac{\left(1-(h(p)+2\varepsilon)(K^2(n,2,-2)+\varepsilon)\right))^{\frac{n}{2}}}{nK(n,2)^n},
\end{equation*}
that is
\begin{equation*}
    \beta^{\frac{2}{n}}\ge{\beta^{*}}^{\frac{2}{n}}-\frac{2\varepsilon^2+\varepsilon(h(p)+2\varepsilon
K^2(n,2,-2))}{n^{\frac{2}{n}}K(n,2)^2}.
\end{equation*}
By assumption $\beta^*>\beta$, if we take  $\varepsilon>0$ small
enough so that
\begin{equation*}
-2\varepsilon^2-\varepsilon(h(p)-2\varepsilon
K^2(n,2,-2))+n^{\frac{2}{n}}K(n,2)^2({\beta^*}^{\frac{2}{n}}-\beta^{\frac{2}{n}})>0,
\end{equation*}
we get a contradiction. Thus $\beta=0$ and \eqref{eqn3.3} assures
that
\begin{equation*}
\int_M|\nabla v_m|^2dv_g=o(1),
\end{equation*}
that is $v_m\to0$ strongly in $H_1^2(M)$.
\end{proof}
\begin{Lemma}\label{lem3.3} Let $v_m$ be  Palais-Smale sequence for $J_h$ at level  $\beta$
that converges weakly and not strongly to $0$ in $H_1^2(M)$. Then,
there exists a sequence of positive reals $R_m\to 0 $ such
that, up to a subsequence
\begin{equation*}
\hat{v}_m(x)=R_m^{\frac{n-2}{2}}\eta_{r}(R_m x))v_m(
\exp_{p}(R_m x)),
\end{equation*}
where $0<r<\frac{\delta_g}{2}$, converges weakly in $D^{1,2}(\mathbb{R}^n)$
to a function $v\in D_1^2(\mathbb{R}^n)$ weak solution of \eqref{0.3}.
\end{Lemma}
\begin{proof} Since the Palais-Smale sequence $v_m$ of $J_h$ at
level $\beta$ converges  weakly and not strongly  in $H_1^2(M)$ to
$0$, then by Lemma \ref{lem3.2}  $\beta\ge\beta^*$.\\
Up to a subsequence, $v_m$ converges strongly to $0$ in
$L_2(M)$.Then, similar computations as in Lemma \ref{lem2.1} give that for all
$\varepsilon>0$ small
$$n\beta^*+o(1)\le\int_M|\nabla v_m|^2dv_g\le\frac{n\beta}{1-(h(p)+\varepsilon)(K^2(n,2,-2)+\varepsilon)}+o(1).$$
In such way that there exist two positive constant such that
\begin{equation}\label{eq3.6}
C_1\le\int_M|\nabla v_m|^2dv_g\le c_2.
\end{equation}
Let $\gamma$ a small positive constant such that
\begin{equation}\label{eq3.7}
    \underset{m\to\infty}{\lim\sup}\int_M|\nabla v_m|^2>\gamma>0.
\end{equation}
Up to a subsequence, for each $m>0,$ we can find the smallest
constant $r_m>0$  such that
\begin{equation}\label{3.11}
\int_{B(p,r_m)}|\nabla v_m|^2dv_g=\gamma.
\end{equation}
Note that for $x\in B(\frac{\delta_g}{2})$, it holds
\begin{equation}\label{}
    dist_g(p,\exp_p(x))\le C_o|x|.
\end{equation}
Let $0<r<\frac{\delta_g}{2}$ and take $R_m$, $ 0<R_m<1$ such that $2C_orR_m\le r_m$, in such way that $\exp_p(R_mB(2r))\subset B(p,r_m)$.\\
Let $x\in
B(R_m^{-1}\delta_g)\subset\mathbb{R}^n$, and define
\begin{eqnarray*}
% \nonumber to remove numbering (before each equation)jjjjjjjj1uuuuuuuuuuuuuuuuuuu
  \hat{v}_m(x) &=& R_m^{\frac{n-2}{2}}\eta_{r}(R_mx)v_m(\exp_{p}(R_m
  x)), \emph{and }
   \\
 \hat{g}_m(x) &=& (\exp_{p}^*g)(R_m x)).
\end{eqnarray*}
We show that the sequence $ \hat{v}$ is bounded in $D^{1,2}(\R^n)$. First we have
\begin{eqnarray}\label{3.13}
\nonumber\int_{B(2r)}|\nabla_{\hat{g}} \hat{v}_m|^2 dv_{\hat{g}}&=&\int_{B(r)}|\nabla_{\hat{g}} \hat{v}_m|^2 dv_{\hat{g}}
+\int_{B(2r)\setminus B(r)}|\nabla_{\hat{g}} \hat{v}_m|^2 dv_{\hat{g}}\\&=& \nonumber
\int_{\exp_p(R_m(B(r))}|\nabla_g v_m|^2 dv_g+\int_{B(2r)\setminus B(r)}|\nabla_{\hat{g}} \hat{v}_m|^2 dv_{\hat{g}}\\ \nonumber &\le&
\gamma+2\int_{\exp_p(R_m(B(2r))\setminus\exp_p(R_m(B(r))}|\nabla_g v_m|^2 dv_g\\ \nonumber &+&\int_{\exp_p(R_m(B(2r))\setminus\exp_p(R_m(B(r))}v_m^2|(\nabla_g \eta)(\exp^{-1}(x))|^2dv_g\\ &\le& 3\gamma+o(1),
\end{eqnarray}
here we have used the strong convergence of $v_m$ to $0$ in $L_2(M)$. Similarly, we can obtain that
\begin{equation}\label{}
    \int_{B(2rR_m^{-1})\setminus B(2r)}|\nabla_{\hat{g}} \hat{v}_m|^2 dv_{\hat{g}}\le C,
\end{equation}
for some positive constant $C$ and thus $\int_{\R^n}|\nabla_{\hat{g}} \hat{v}_m|^2 dv_{\hat{g}} $ is bounded.\\
On the other hand, since $\hat{g}_m$ goes smoothly to the Euclidean metric on $\R^n$, we can find a constant $0<C<1$ such that for $m$  large and $u$ such that $supp u\subset B(2rR_m^{-1})$, it holds
\begin{equation}\label{}
  \int_{\R^n}|\nabla u|^2dx \le C\int_{\R^n}|\nabla_{\hat{g}} u|^2 dv_{\hat{g}},
\end{equation}
thus, we get that $\hat{v}_m$ is bounded in $D^{1.2}(R^n)$.\\
 Consequently, up to a subsequence, $\tilde{v}_m$ converges weakly to some
function $v\in D^{1,2}(\R^n)$.\\
Suppose that $v\neq0$,  we show that $v$ is a weak solution on $D^{1,2}(\R^n )$
to \eqref{0.3}. First, notice that since the sequence $v_m$ converges strongly in $L_2(M)$
to $0$ and the sequence $\hat{v}_m$ converges strongly in $L_{2}(B(2r))$ to $v\neq0$, it follows that $R_m\to 0$.\\
Let $\varphi\in\C^\infty_o(\mathbb{R}^n)$ be a function with compact
support included in the ball $B(2r)$. For $m$ large, define
on $M$ the sequence $\varphi_m$ as
$$\varphi_m(x)=R_m^{\frac{2-n}{2}}\varphi(R_m^{-1}(\exp^{-1}_{p}( x))).$$
Then, $\varphi_m$ is bounded in $H^2_1(M)$ and
\begin{equation}\label{}
   \int_M\nabla v_m\nabla \varphi_m dv_g+R_m^n\int_M\nabla\eta\nabla\varphi_m v_mdv_g=\int_{\mathbb{R}^n}\nabla\tilde{v}_m\nabla\varphi dv_{\hat{g}_m},
\end{equation}
by the strong convergence of $v_m$ in $L_2(M)$ to $0$, after doing a Holder inequality, the second term of the left-hand side converges to $0$. Then we obtain
\begin{equation}\label{}
   \int_M\nabla v_m\nabla \varphi_m dv_g=\int_{\mathbb{R}^n}\nabla\tilde{v}_m\nabla\varphi dv_{\hat{g}_m}+o(1),
\end{equation}
Moreover,
\begin{eqnarray*}
   \int_M\frac{h_m}{\rho_p^2}v_m\varphi_m dv_g &=&
  R_m^{2}\int_{\mathbb{R}^n}\frac{h_m(\exp_{p}(R_m x) )}
  {dist_{\hat{g}_m}^2(0, R_m x)}\tilde{v}_m\varphi dv_ {\hat{g}_m}, \emph{ and }\\
   \int_M |v_m|^{2^*-2}v_m\varphi_m
   dv_g&=&\int_{\mathbb{R}^n}|\tilde{v}_m|^{2^*-2}\tilde{v}_m\varphi
   dv_ {\hat{g}_m}.
\end{eqnarray*}
Since $v_m$ is a Palais-Smale sequence of $J_h$, by passing to the limit when
$m\to\infty$,  we get that $v$ is weak solution  of \eqref{0.3}.
\end{proof}
\begin{Lemma}\label{lem3.4} Let $v$ be the solution  of \eqref{0.3} given by Lemma \ref{lem3.3} and such that $v\neq0$, then
up to a subsequence,
\begin{equation*}
w_m=v_m-R_m^{\frac{2-n}{2}}\eta_r(\exp^{-1}(x))
v(R_m^{-1}\exp_{p}^{-1}(x)),
\end{equation*}
 where
$0<r<\frac{\delta_g}{2}$,  is a Palais-Sequence for $J_h$
 that weakly converges to $0$ in
$H_1^2(M)$ and $J_h(w_m)=J_h(v_m)-J_\infty(v)$.
\end{Lemma}
\begin{proof}
 For $0<r<\frac{\delta_g}{2}$, put
\begin{equation*}
\B_m(x)=R_m^{\frac{2-n}{2}}\eta_r(\exp_p^{-1}(x))
v(R_m^{-1}\exp_p^{-1}( x)),x\in M
\end{equation*}
in such way that
\begin{equation*}
w_m=v_m-\B_m.
\end{equation*}
We begin by proving that  $\B_m$  converges weakly to $0$ in
$H^2_1(M)$ and thus does $w_m$.\\
Take a function $\varphi\in\C^{\infty}(M)$, then we have
\begin{eqnarray*}
&\int_{B(p,2r)}\left(\nabla \B_m\nabla\varphi+
\B_m\varphi\right)dv_g&\\&=
R_m^{\frac{n}{2}}\int_{B(2r R_m^{-1})}[R_m
v(x)(\nabla\eta_r)(R_m x)+\eta_r(R_m x)\nabla
v]\nabla \varphi(\exp_p(R_m x))
dv_{\hat{g}_m}&\\&+R_m^{\frac{n+2}{2}}\int_{B(2r
R_m^{-1})} v\eta_r(R_m x)\varphi(\exp_p(R_m
x))dv_{\hat{g}_m},&
\end{eqnarray*}
then, for a positive constant $C'$ such that $dv_{\hat{g}_m}\le
C'dx$, it follows that
\begin{eqnarray*}
&\int_{B(p,2r)}\left(\nabla \B_m\nabla\varphi+
\B_m\varphi\right)dv_g\\&\le
C'R_m^{\frac{n}{2}}[\sup_M|\nabla\varphi|\int_{\R^n}(|\nabla
v|+|v|Cr^{-1})dx+ R_m\sup_M|\varphi|\int_{\R^n}|v|)dx].&
\end{eqnarray*}
 Thus, when tending $m\to \infty$, we ge that $\B_m\to0 $ weakly in $H^2_1(M)$.\\
Now, let us evaluate $J_h(w_m)$. First, we have
\begin{eqnarray*}
\int_{M}|\nabla w_m|^2dv_g&=&\int_{M\setminus
B(p,2r)}|\nabla v_m|^2dv_g+\int_{B(p,2r)}|\nabla
(v_m-\B_m)|^2dv_g,
\end{eqnarray*}
and of course
\begin{eqnarray*}
  &&\int_{B(p,2r)}|\nabla
(v_m-\B_m)|^2dv_g\\ &=& \int_{B(p,2r)}|\nabla
v_m|^2dv_g-2\int_{B(p,2r)}\nabla v_m\nabla\B_m
dv_g+\int_{B(p,2r)}|\nabla \B_m|^2dv_g.
\end{eqnarray*}
Direct calculation gives
\begin{eqnarray*}
&\\&\int_{B(p,2r)}|\nabla \B_m|^2dv_g=\int_{B( 2r
R_m^{-1})}\eta^2_r(R_m x)|\nabla
v|^2dv_{\hat{g}_m}+&\\&R_m^2\int_{B( 2r
R_m^{-1})}v^2|\nabla\eta_{r }|^2(R_m
x)dv_{\hat{g}_m}+2R_m\nabla\eta_{r }(R_m
x)\nabla vdv_{\hat{g}_m}.&
\end{eqnarray*}
It can be easily seen that the second term of right-hand side member
of the above equality tends to $0$ as $m\to \infty $.
Furthermore,  for $R>0$, a positive constant, we write
\begin{equation*}
\int_{B( 2r R_m^{-1})}\eta^2_r(R_m x)|\nabla
v|^2dv_{\hat{g}_m}=\int_{B( R)}\eta^2_r(R_m x)|\nabla
v|^2dv_{\hat{g}_m}+\int_{\mathbb{R}^n\setminus
B(R)}\eta^2_r(R_m x)|\nabla v|^2dv_{\hat{g}_m}.
\end{equation*}
 with
\begin{equation*}
\int_{\mathbb{R}^n\setminus B(R)}\eta^2_r(R_m x)|\nabla
v|^2dv_{\hat{g}_m}\le C \int_{\mathbb{R}^n\setminus
B(R)}|\nabla v|^2dx=\varepsilon_R,
\end{equation*}
where $\varepsilon_R$ is a function in $R$ such that $ \varepsilon_R\to 0 $ as $R\to\infty$.\\
Noting that, that ${\hat{g}_m}$ goes locally in $C^1$ to the
Euclidean metric $\xi$, we get then
\begin{equation}
\int_{B(p,2r)}|\nabla
\B_m|^2dv_g=\int_{\mathbb{R}^n}|\nabla
v|^2dx+o(1)+\varepsilon_R.
\end{equation}\label{3.15}
Moreover, we have
\begin{eqnarray}\label{3.16}
&\int_{B(p,2r)}\nabla v_m\nabla\B_m
dv_g=\int_{B(2r R_m^{-1})}\nabla
\tilde{v}_m\nabla vdv_{\hat{g}_m}
&\\&\nonumber+R_m\int_{B(2r R_m^{-1})}(R_mv\nabla
\hat{v}_m -\tilde{v}_m\nabla v)(\nabla\eta_{r})(R_m
x) dv_{\hat{g}_m},&
\end{eqnarray}
with $ \tilde{v}_m=R_m^{}\frac{n-2}{2}v_m(\exp_p(R_mx))$.\\
We have
\begin{eqnarray*}
 && |\int_{B(2r
R_m^{-1})}(\nabla\eta_r)(R_m x)(R_mv\nabla \tilde{v}_m
-\tilde{v}_m\nabla v)dv_{\hat{g}_m}|\\
 &\le&
cr^{-1}\left[\right.R_m\int_{B(2r R_m^{-1})} |\nabla
\tilde{v}_m|^2 dv_
{\hat{g}_m})^{\frac{1}{2}}(\int_{B(2r R_m^{-1})} v^2
dx)^{\frac{1}{2}}\\  &+&(\int_{B(2r
R_m^{-1})}\tilde{v}_m^2
dv_{\hat{g}_m})^{\frac{1}{2}}(\int_{B(2r R_m^{-1}))}
|\nabla v|^2 dx)^{\frac{1}{2}}\left.\right].
\end{eqnarray*}
Since $v_m$ is bounded in $H^2_1(M)$, the quantities
$\int_{B(2r R_m^{-1})} |\nabla \hat{v}_m|^2 dv_
{\hat{g}_m}$ and$\int_{B(2r R_m^{-1})} |
\hat{v}_m|^2 dv_ {\hat{g}_m}$ are bounded and hence the
second term of the right-hand side member of \eqref{3.16} is $o(1)$.
Thus, by using the weak convergence of
$\hat{v}_m$ to $v$ in $D^{1,2}(\R^n)$ that
\begin{equation*}
\int_{B(p,r)}\nabla v_m\nabla\B_m
dv_g=\int_{\mathbb{R}^n}|\nabla v|^2dx+o(1).
\end{equation*}
so that
\begin{equation*}
\int_{M}|\nabla w_m|^2dv_g=\int_{M}|\nabla
v_m|^2dv_g-\int_{\mathbb{R}^n}|\nabla
v|^2dx+o(1)+\varepsilon_R.
\end{equation*}
In similar way, for $R$ a positive constant and $m$ large, we
write
\begin{equation*}
\int_{B(p,2r)}\frac{h}{\rho_p^2}\B_m^2dv_g=
\int_{B(p,RR_m)}\frac{h}{\rho_p^2}\B_m^2dv_g+
\int_{B(p,2r)\setminus
B(p,RR_m)}\frac{h}{\rho_p^2}\B_m^2dv_g
\end{equation*}
with
\begin{eqnarray*}
\int_{B(p,2r)\setminus
B(p,RR_m)}\frac{h}{\rho_p^2}\B_m^2dv_g&\le& C
(RR_m)^{-2}\int_{B(p,2r)\setminus
B(p,RR_m)}\B_m^2dv_g\\&=&C
R^{-2}\int_{\R^n\setminus B(R)}v^2dx=\varepsilon_R
\end{eqnarray*}
with $\varepsilon_R\to0$ as $R\to\infty$.\\
Hence, when letting $m\to\infty$ and $R\to\infty$, we get
\begin{eqnarray*}
\int_{B(p,2r)}\frac{h}{\rho_p^2}\B_m^2&=&
R_m^2\int_{B(R)} \frac{h(\exp_p(R_m
x))}{(dist_{\hat{g}_m}(0,R_m x))^2}
\eta^2_r(R_m x) v^2dv_{\hat{g}_m}+\varepsilon_R\\
   &=&h(p)\int_{\mathbb{R}^n}\frac{v^2}{|x|^2}dx+o(1).
\end{eqnarray*}
Also, in similar way, since $v_m$ is bounded in $H^2_1(M)$,
after using Hölder and Hardy inequalities, we can easily have
\begin{equation*}
\int_{B(p,2r)\setminus
B(p,RR_m)}\frac{h}{\rho_p^2}v_m\B_m dv_g\le C
R^{-2}\int_{\R^n\setminus B(R)}v^2dv_g=\varepsilon_R,
\end{equation*}
which yields
\begin{eqnarray*}
\int_{B(p,r)}\frac{h}{\rho_p^2}v_m\B_m dv_g&=&
R_m^2\int_{B(R)} \frac{h(\exp_p(R_m
x))}{(dist_{\hat{g}_m}(0,R_m x))^2}
\hat{v}_m vdv_{\hat{g}_m}+\varepsilon_R\\
   &=&h(p)\int_{\mathbb{R}^n}\frac{v^2}{|x|^2}dx+o(1).
\end{eqnarray*}
so that in the end we obtain
\begin{equation*}
\int_M\frac{h}{\rho_p^2}w_m^2dv_g=
\int_M\frac{h}{\rho_p^2}v_m^2dv_g-h(p)\int_{\mathbb{R}^n}\frac{v^2}{|x|^2}dx+o(1).
\end{equation*}
In similar way, we can prove that
\begin{equation*}
\int_M|w_m|^{2^*}dv_g=\int_M|v_m|^{2^*}dv_g-\int_M|v|^{2^*}dx+o(1),
\end{equation*}
Finally, summing up all the calculations, we obtain
\begin{equation*}
J_h(w_m)=J_m(v_m)-J_\infty(v)+o(1).
\end{equation*}
It remains to prove that  $DJ_h(\B_m)\to 0$ in
$H_1^2(M)'$. Let $\varphi\in
 H_1^2(M)$, for $x\in B(r R_m^{-1})$ put $\varphi_m(x)=R_m^{\frac{n-2}{2}}\varphi(\exp_p(R_m
x))$ and $\overline{\varphi}_m(x)=\eta_r(R_m
x))\varphi_m (x)$, then we have
\begin{eqnarray*}
\int_{B(p,2r)} \nabla \B_m\nabla\varphi
dv_g=\int_{B(2r R_m^{-1})} \nabla
v\nabla\overline{\varphi}_m
dv_{\hat{g}_m}&&\\+R_m\int_{B(2r R_m^{-1})}
\nabla\eta_r(R_m
x)(v\nabla\varphi_m-\varphi_m\nabla v)dv_{\hat{g}_m}.
  \end{eqnarray*}
Knowing that $\int_{B(p,2r)} |\nabla\varphi|^2
dv_g=\int_{B(2r R_m^{-1})}  |\nabla\varphi_m|^2
dv_{\hat{g}_m}$, we get that
\begin{equation*}
\int_{B(2r R_m^{-1})}| \nabla\eta_r(R_m
x)(v\nabla\varphi_m-\varphi_m\nabla
v)|dv_{\hat{g}_m}\le C||\varphi||_{H^2_1(M)},
\end{equation*}
which gives that
\begin{equation*}
\int_{B(p,2r)} \nabla \B_m\nabla\varphi
dv_g=\int_{B(2r R_m^{-1})} \nabla
v\nabla\overline{\varphi}_m dv_{\hat{g}_m}+o(||\varphi
||_{H_1^2(M)}).
\end{equation*}
Next, for $R>0$ write
\begin{equation*}
\int_{B(2r R_m^{-1})} \nabla
v\nabla\overline{\varphi}_m dv_{\hat{g}_m}=\int_{B(R)}
\nabla v\nabla\overline{\varphi}_m
dv_{\hat{g}_m}+\int_{B(2r R_m^{-1})\setminus B(R)}
\nabla v\nabla\overline{\varphi}_m dv_{\hat{g}_m},
\end{equation*}
note that
\begin{eqnarray*}
  \int_{B(2r R_m^{-1})\setminus B(R)}
\nabla v\nabla\overline{\varphi}_m dv_{\hat{g}_m}&\le&
C||\varphi||_{H^2_1(M)}(\int_{B(2r R_m^{-1})\setminus
B(R)} |\nabla v|^2 dx)^\frac{1}{2}\\&=& O(||\varphi||_{H^2_1(M)}
)\varepsilon(R),
\end{eqnarray*}
where  $\varepsilon_R\to 0$ as $R\to \infty$. Since the sequence of
metrics ${\hat{g}_m} $ tends locally in $C^1$ when
$m\to\infty$ to the Euclidean metric, we obtain
\begin{equation*}
\int_{B(p,2r)} \nabla \B_m\nabla\varphi dv_g=\int_{\R^n}
\nabla v\nabla\overline{\varphi}_m dx+o(||\varphi
||_{H_1^2(M)})+O(||\varphi||_{H^2_1(M)} )\varepsilon(R).
\end{equation*}
Moreover, for a given  $R>0$, we have for $m$ large,
\begin{equation*}
\int_{B(p,2r)}\frac{h}{\rho_p^2}\B_m\varphi dv_g=
\int_{B(p,RR_m)}\frac{h}{\rho_p^2}\B_m\varphi
dv_g+\int_{B(p,2r)\setminus
B(p,RR_m)}\frac{h}{\rho_p^2}\B_m\varphi dv_g.
\end{equation*}
 On the one hand, we have
\begin{equation*}
\int_{B(p,2r)\setminus
B_p(RR_m)}\frac{h}{\rho_p^2}\B_m\varphi dv_g\le
\frac{C}{(RR_m)^2}||\varphi||_{H_1^2(M)}\int_{B(p,2r)\setminus
B(p,RR_m)}\B_m^2dv_g,
\end{equation*}
and a straightforward computation  shows that
\begin{equation*}
\int_{B(p,2r)\setminus B(p,RR_m)}|\B_m|^2dv_g\le C
R_m^2\int_{B(2r R_m^{-1})\setminus B(R)}v^2dx,
\end{equation*}
  which implies that
\begin{equation*}
   \int_{B(p,2r)\setminus
B(p,RR_m)}\frac{h}{\rho_p^2}\B_m\varphi dv_g=O(
||\varphi||_{H_1^2(M)})\varepsilon_R
\end{equation*}
with
$\varepsilon_R \to 0$ as $R\to \infty$. \\
On the other hand,  we have
\begin{equation*}
\int_{B(p,RR_m)}\frac{h}{\rho_p^2}\B_m\varphi dv_g=
R_m^{2}\int_{B(R)}\frac{h(\exp_pR_m
x)}{(dist_{\hat{g}_m}(0,R_m
x))^2}v\overline{\varphi}dv_{\hat{g}}.
\end{equation*}
which leads to
\begin{eqnarray*}
  \int_{B_p(RR_m)}\frac{h}{\rho_p^2}\B_m\varphi
dv_g &=&\int_{B(R)}\frac{h(p)
}{|x|^2}v\overline{\varphi}dx+o(||\varphi ||_{H_1^2(M)}) \\
   &=&\int_{\mathbb{R}^n}\frac{h(p)
}{|x|^2}v\overline{\varphi}dx-\int_{\mathbb{R}^n\setminus
B(R)}\frac{h(p) }{|x|^2}v\overline{\varphi}dx +o(||\varphi
||_{H_1^2(M)}),
\end{eqnarray*}
with
\begin{eqnarray*}
\int_{\mathbb{R}^n\setminus B(R)}\frac{h(p)
}{|x|^2}v\overline{\varphi}dx&\le&
\frac{C}{R^2}||\varphi||_{H_1^2(M)}\\ &=&O(
||\varphi||_{H_1^2(M)})\varepsilon_R.
\end{eqnarray*}
so that \begin{equation*}
   \int_{B(p,2r)}\frac{h}{\rho_p^2}\B_m\varphi dv_g= \int_{\mathbb{R}^n}\frac{h(p)
}{|x|^2}v\overline{\varphi}dx+o(||\varphi ||_{H_1^2(M)})+O(
||\varphi||_{H_1^2(M)})\varepsilon_R.
\end{equation*}
In the same way, we can also have
\begin{equation*}
    \int_{B(p,2r)}|\B_m|^{\frac{4}{n-2}}\B_m\varphi dv_g =
\int_{\mathbb{R}^n}|v|^{\frac{4}{n-2}}v\overline{\varphi}_m
dx+o(||\varphi ||_{H_1^2(M)}) +O(||\varphi
||_{H_1^2(M)})\varepsilon_R.
\end{equation*}
Summing up, we obtain
\begin{eqnarray*}
&\int_{B(p,2r)}(\nabla\B_m\nabla \varphi
dv_g+\frac{h}{\rho_p^2}\B_m\varphi)dv_g-
\int_{B(p,2r)}f|\B_m|^{\frac{4}{n-2}}\B_m\varphi
dv_g& \\&= \int_{\R^n}( \nabla v\nabla\overline{\varphi}_m
dx+\frac{h(p) }{|x|^2}v\overline{\varphi}_m)dx-
  f(p)\int_{\mathbb{R}^n}|v|^{\frac{4}{n-2}}v\overline{\varphi}_m dx&\\&+o(||\varphi
||_{H_1^2(M)}) +O(||\varphi ||_{H_1^2(M)})\varepsilon_R,&
\end{eqnarray*}
and since $v$ is  weak solution of \eqref{0.1}, we get the desired
result.
\end{proof}
\begin{Lemma}\label{lem3.5} Let $v_m$ be a Palais-Smale sequence for $J_h$.
Suppose that the sequence
$\hat{v}_m$ of the above lemma
converges weakly to $0$ in $D^{1,2}(\R^n)$. Then, there exist a
sequence of positive numbers $\{\t_m\},\t_m\to 0$ and a
sequence of points $x_i\in M, x_i\to x_o\neq p$ such
that up to a subsequence, the sequence
\begin{equation*}
 \nu_m=\t_m^{\frac{n-2}{2}}\eta_r(\t_mx)v_m(\exp_{x_i}(\t_m x ))
\end{equation*}
 converges weakly to a nontrivial weak solution $\nu$ of the Euclidean
equation \eqref{0.2} and the sequence
\begin{equation*}
    \W_m=v_m-\t_m^{\frac{2-n}{2}}\eta_r(\exp_{x_i}^{-1}(x))\nu(\t_m^{-1}\exp_{x_i}^{-1}(x))
\end{equation*}
is a Palais-Smale sequence for $J_h $ that converges weakly to $0$ in
$H^1_2(M) $ and
\begin{equation*}
J_h(\W_m)=J_h(v_m)-J(\nu).
\end{equation*}
\end{Lemma}
\begin{proof}  Take a function $\varphi\in
\C^{\infty}_o(B(C_or))$ and put
$\varphi_m(x)=\varphi(R_m^{-1}\exp_p^{-1}(x)) $. By the strong convergence of
$\hat{v}_m$ to $0$ in $L_{loc}^2(\R^n)$, we have for $m$
large
\begin{eqnarray}\label{3.20}
 &&\int_{\R^n}|\nabla( \hat{v}_m\varphi)|^2dv_{\hat{g}_m}=\int_{\R^n}\nabla
\hat{v}_m\nabla(\hat{v}_m\varphi^2)dv_{\hat{g}_m}+o(1)
 \\ &=&\int_{M}\nabla v_m\nabla
(v_m\varphi_m^2)dv_{g}+o(1)\nonumber\\ &=&\|\mathcal{D}
J_h\| \|{v_m\varphi_m^2}\|+
  \int_{M}\frac{h}{\rho_p^2}(v_m\varphi_m)^2dv_g+
  \int_{M}|v_m|^{\frac{4}{n-2}}(v_m\varphi_m)^2dv_g+o(1)\nonumber\\
   &\le&(h(p)+\varepsilon)(K^2(n,2,-2)+\varepsilon)\int_{\R^n}|\nabla
  (\tilde{v}_m\varphi)|^2dv_{\hat{g}_m}+\nonumber\\ && K^{2^*}(n,2)(\int_{B(C_or)}|\nabla \hat{v}_m|^2dv_{\hat{g}_m})^{\frac{2}{n-2}}\int_{\R^n}|\nabla
   (\hat{v}_m\varphi)|^2dv_{\hat{g}_m}+o(1)\nonumber.
\end{eqnarray}
Thus, for $\gamma$ (in \eqref{3.11}) chosen small enough,  we get that for each
$t,0<t<C_or,$
\begin{equation}\label{4.20}
\int_{B(p,tR_m)}|\nabla v_m|^2dv_g=\int_{B(t)}|\nabla
\tilde{v}_m|^2dv_{\hat{g}}\to0, \text{ as }m\to \infty.
\end{equation}
Now, for $t>0$ consider the function
\begin{equation*}
t\longrightarrow\F(t)=\max_{x\in M}\int_{B(x,t)}|\nabla
v_m|^2dv_g.
\end{equation*}
Since $\F $ is continuous, under \eqref{eq3.6} and \eqref{eq3.7}, it
follows that for any $\lambda\in(0,\gamma)$, there exist
$t_m>0$ small and $x_m\in M$ such that
\begin{equation}\label{3.21}
\int_{B(x_m,t_m)}|\nabla v_m|^2dv_g=\lambda.
\end{equation}
Since $M$ is compact, up to a subsequence, we may assume that
$x_m$ converges to some point $x_o\in M$.\\
Note first that for all  $m\ge0$,
$t_{m}<r_{m}=C_orR_{m}$, otherwise if there exists
$m_o\ge 0$ such that $t_{m_o}\ge r_{m_o}$, we get a
contradiction due to the fact that
\begin{equation*}
\lambda= \int_{B(x_{m_o},t_{m_o})}|\nabla
v_{m_o}|^2dv_g\ge\int_{B(p,t_{m_o})}|\nabla
v_{m_o}|^2dv_g\ge\int_{B(p,r_{m_o})}|\nabla
v_{m_o}|^2dv_g =\gamma
\end{equation*}
and $\lambda$ is chosen such that $0<\lambda<\gamma$. Since $r_m\to0$, it follows that $t_m\to0$ as $m\to\infty$. \\ Now, suppose that for all $\varepsilon>0$, there exists
$m_\varepsilon>0$ such that $dist_g(x_m,p)\le\varepsilon$
 for all $m\ge m_\varepsilon$. Choose $r'_m$ such that,  $t_{m}<r'_m<r_{m} $ and take $\varepsilon'
=r'_m-t_m$, we get that for some $m_{\varepsilon'}>0$
and $m\ge m_{\varepsilon'}$
\begin{equation*}
 B(x_m, t_m)\subset B(p,r'_m),
\end{equation*}
which, by virtue of \eqref{4.20} and \eqref{3.21}, is impossible. We deduce then that
$x_o\neq p$. \\
For $y$ and $z\in B(\frac{\delta_g}{2})$ we can find a positive constant $C_1$ such that  for all $x\in M$
\begin{equation}\label{}
dist_g(\exp_x(y),\exp_x(z)\le C_1|y-z|.
\end{equation}
Take $\t_m$ such that $C_1r\t_m=t_m$. Then, for $a\in\R^n$ and $r>0$ a constant such that $|a|+r<\tau^{-1}\delta_g$, we have
\begin{equation}\label{3.22}
 \exp_{x_m}(\t_m B(a,r))\subset B(\exp_{xm}(\t_m
a),C_1r\t_m),
\end{equation}
and
\begin{equation*}
   \exp_{x_m}(\t_m B(C_1r))= B(x_m,C_1r\t_m).
\end{equation*}
For $x \in B(\t_m^{-1}\delta_g)\subset\mathbb{R}^n $ consider the
sequences
\begin{eqnarray*}
% \nonumber to remove numbering (before each equation)jjjjjjjj1uuuuuuuuuuuuuuuuuuu
  \nu_m(x) &=&\t_m^{\frac{n-2}{2}}\eta_{r}(\t_m x)v_m(\exp_{x_m}(\t_m
  x)), \text{ and }
   \\
 \tilde{g}_m(x) &=& \exp_{x_m}^*g(\t_m x)).
\end{eqnarray*}
 As in the proof of the above
lemma, we can easily check that there is a subsequence of
$\nu_m$ that weakly converges  in
$\dob(\R^n)$ to some function $\nu$, a weak solution on $\dob(\R^n)$
to \eqref{0.2}. Note that this time the
singular term disappears because $x_o\neq p$  and because of course $t_m\to0$.\\
It remains to show that $\nu\neq0$. For this purpose, take a point
$a\in\R^n$ and a constant $r>0$ such that $|a|+r<\delta_g\t_m^{-1}$.
Since we have
\begin{equation*}
\int_{B(a,r)}|\nabla
\nu_m|^2dv_{\tilde{g}_m}=\int_{\exp_{x_m}(\t_m
B(a,r))}|\nabla v_m|^2dv_{g},
\end{equation*}
we get by construction of $x_m$ and \eqref{3.22} that for such $a$ and $r$ ,
\begin{equation*}
\int_{B(a,r)}|\nabla \nu_m|^2dv_{\tilde{g}}\le\lambda.
\end{equation*}
Suppose now that $\nu\equiv0$. Take any function $\varphi\in \dob(\R^n)$
with support included in a ball $B(a,r)\subset\R^n$, with $a$ and
$r$ as above. Then, by taking $\lambda$ small enough, we get by the
same calculation done in \eqref{3.20} that $\int_{B(a,r)}\nabla
\nu_m dv_{\tilde{g}}$ converges to $0$ for all $a\in\R^n$
and $r>0$ such that  $|a|+r<\t_m^{-1}\delta_g$. In particular, for $r$ small such that $C_1r<\t_m^{-1}\delta_g$, we get
\begin{equation*}
\int_{B(x_m,t_m)}|\nabla
v_m|^2dv_{g}=\int_{B(C_1r)}|\nabla
\nu_m|^2dv_{\tilde{g}}\to0,
\end{equation*}
which makes a contradiction. Thus $\nu\neq0$.\\
The proof of the remaining statements of the lemma goes in the same
way as in lemma \ref{lem3.5}.
\end{proof}
\begin{proof}[Proof of theorem \ref{thm3.6}] Let $u_m$ be a Palais-Smale sequence for $J_h$ at level $\beta$. It can easily seen that  $u_m$ is bounded in $H^2_1(M)$, then  up to a subsequence it converges to a function $u$ weakly in $H^2_1(M),L_2(\rho_p^{-2},M)$, strongly in $L_q(M),q<2^{*}$ and almost everywhere to $u$ in $M$.\\
Thus, by Lemma \ref{lem2.1}, the function $u$ is weak solution of \eqref{0.1} and the sequence $v_m=u_m-u$ is a Palais-Smale sequence for $J_h$ at level $\beta-J_h(u)$. If $v_m$ converges strongly to 0 in $H^1_2(M)$, then the theorem is proved with $k=l=0$. If not,  then by lemma \ref{lem3.3}, there exists a sequence of reals $R_m\to 0$ such that, the sequence $\hat{v}_m$
\begin{equation*}
\hat{v}_m(x)=R_m^{\frac{n-2}{2}}\eta_r(R_m x)v_m(\exp_p(R_m x)), x\in \R^n
\end{equation*}
converges in $D^{1,2(\R^n)}$ to a solution $v$ of \eqref{0.3}.\\
 If $v\neq0$, then  by lemmas \ref{lem3.4} and \ref{lem2.1} the sequence
\begin{equation*}
v_{1,m}(x)=v_m(x)-R_m^{\frac{2-n}{2}}\eta_r(\exp_p^{-1}(x)v(R_m\exp_p^{-1}( x))
\end{equation*}
is a Palais-smale sequence for $J_h$ at level $\beta_1=\beta-J_\infty(u)-J(v)$ that converges weakly to $0$ in $H^1_2(M)$. If $v_{1,m} $ converges strongly, the theorem is proved with $k=1$ and $l=0$. If not, we repeat  to $v_{1,m}$ the procedure already applied to $v_m$ to obtain a Palais-Smale sequence $v_{2,m}$ at level $\beta_2<\beta_1$ that either converges strongly to $ 0$ in $H^1_2(M)$ and in this case the theorem is proved with $k=2, l=0$ or it converges only weakly, and in this case we repeat again the  procedure applied to $v_{1,m}$. We keep repeating this proceeding until we get a Palais-Smale sequence at level
\begin{equation*}
\beta_k<\beta^*=\frac{1-h(p)K(n,2,-2)^2}{nK(n,2)^n},
\end{equation*}
 and int his case, by lemma \ref{lem3.2} the Palais-Smale sequence associated to the level $\beta_k$ converges strongly to $0$.\\
If $v=0$, then by lemma \ref{lem3.5}, there exists a sequence of positive reals $\tau_m\to0$ and a sequence of points $x_i\to x_o\neq p$ such that the sequence
\begin{equation*}
    \hat{w}_m(x)=\tau_m^{\frac{n-2}{2}}\eta_r(\tau_m x)v_m(\exp_{x_i}(\tau_m x)), x\in \R^n
\end{equation*}
converges, up to a subsequence to a solution $\hat{v}\in D^{1,2(\R^n)}\setminus\{0\} $ of equation \eqref{0.2}
and the sequence
\begin{equation*}
\hat{w}_{m,1}(x)=v_m(x)-\tau_m^{\frac{2-n}{2}}\eta_r(\exp_{x_i}^{-1}(x)\hat{v}(\tau_m\exp_{x_o}^{-1}( x))
\end{equation*}
is a Palais-Smale sequence of the functional $J_h$ at lower level $\hat{\beta}_1=\beta-J(u)-G(\hat{v})$ that converges weakly to $0$ $ \in D^{1,2(\R^n)}$. If $\hat{\beta}_1<\beta^*$, then the sequence $\hat{w}_{1,m}(x)$ converges strongly and the theorem is proved with $k=0, \ell=1$, otherwise we apply the procedure from the beginning of the proof to $\hat{w}_{1,m}(x)$ to obtain a palais-Smale sequence $\hat{w}_{2,m}(x)$ at much lower level. We keep doing this procedure so on  until we obtain a Palais-Smale sequence  $\hat{w}_{\ell,m}(x)$ at level $\hat{\beta}_\ell<\beta^*$ that converges strongly to $0$  in  $H^2_1(M)$.
\end{proof}
As it will be shown in the following corollary, the conclusion of the above theorem is very useful in obtaining levels for which Palais-Smale sequence of $J_h$ converges to non zero critical points of the functional $J_h$. \\
First,  put
\begin{eqnarray}
&&\label{2.8}d^*=\frac{1}{nK(n,2)^n},\\&&\label{2.9} D^*=\frac{(1-K(n,2,-2)^2h(p))^{\frac{n-1}{2}}}{nK(n,2)^n},
\end{eqnarray}
By taking $\lambda=h(p)$ in equation \eqref{2.2}, we deduce, by \eqref{0.4}, that if $u$ is a constant sign solution, then
\begin{eqnarray*}
 J_\infty(u) =D^*.
\end{eqnarray*}
On the other hand, if $u$  changes sign, then
\begin{eqnarray*}
 J_\infty(u) >2D^*.
\end{eqnarray*}
In fact,  write $u=u^++u^-$, where $u^+=\max(u,0)$ and $v^-=\min(u,0)$. We then get
\begin{eqnarray*}
\int_{\R^n}\left(|\nabla u^+|^2-\frac{h(p)}{|x|^2}{u^{+}}^{2}\right)dx&=&\int_{\R^n}(\nabla u.\nabla u^+-\frac{h(p)}{|x|^2}uu^{+})dx\\&=&\int_{\R^n}|u|^{2^*-1}uu^+dx\\&=&\int_{\R^n}|u^+|^{2^*}dx
\end{eqnarray*}
Then, since $u^+$ cannot be one of the functions $U_\lambda$, where $U_\lambda is$ defined by \eqref{1.2'},  then by \eqref{0.4} we get
\begin{eqnarray*}
J_\infty(u^+)&=&\frac{1}{n}\int_{\R^n}\left(|\nabla u^+|^2-\frac{h(p)}{|x|^2}{u^{+}}^{2}\right)dx\\&>&\frac{(1-K(n,2,-2)^2h(p))^{\frac{n-1}{2}}}{nK(n,2)^n}=D^*
\end{eqnarray*}
By the same way, we get
\begin{equation*}
J_\infty(u^-)=\frac{1}{n}\int_{\R^n}\left(|\nabla u^-|^2-\frac{h(p)}{|x|^2}{u^{-}}^{2}\right)dx>D^*
\end{equation*}
Thus, we obtain
\begin{equation*}
    J_\infty(u)=J_\infty(u^+)+J_\infty(u^-)>2D^*.
\end{equation*}
Now, define
\begin{equation}
    \mu=\inf_{u\in H^2_1(M)\setminus\{0\}}\frac{\int_M(|\nabla u|^2-\frac{h}{\rho_p^2})dv_g}{(\int_M|u^{2^*}|dv_g)^{\frac{2}{2^*}}}.
\end{equation}
We prove the following corollary:
\begin{Corollary}\label{coro1} Let $u_m$ be a Palais-Smale sequence of $J_h$ at level $\beta$ such that
\begin{equation}\label{3.27}
0<\beta<D^*, \text{ where } D^* \text{ is defined by } \eqref{2.9}.
\end{equation}
Then, up to a subsequence, $u_m$ converges strongly in $H^2_1(M)$ to a function $u\neq0$ such that $DJ_h(u)=0$.\\
Moreover, suppose that
\begin{eqnarray}\label{3.28}
&&h(p)>0, 0<\left(1-h(p)K^{2}(n,2,-2)\right)^{\frac{n-1}{2}}<\frac{1}{2},\\&& \label{3.29}\text{ and  }\mu^{\frac{n}{2}}\ge nD^*.
\end{eqnarray}
Then, if
\begin{equation}\label{3.30'}
D^*<\beta<2D^*,
\end{equation}
 the sequence $u_m$ converges, up to a subsequence, strongly in $H^2_1(M)$ to a function $u\neq0$ such that $DJ_h(u)=0$.
\end{Corollary}
\begin{proof} By theorem \ref{thm3.6}, there is a critical point  $u$ of $J_h$, a sequence of reals, $R_m\to0,\tau_m\to0$, a sequence of points $x_j\to x_o\neq p$, a sequence of solutions $v_i$ of \eqref{0.3} and sequence of non trivial solutions  $\nu_j$ of \eqref{0.2} such that, up to a
subsequence of $u_m$, we have
\begin{eqnarray}\label{3.30}
\nonumber u_m&=&u+\sum_{i=1}^{k}(R^i_m)^{\frac{2-n}{n}}
\eta_r(\exp^{-1}_p(x))v_i((R_m^i)^{-1}\exp^{-1}_p(x))\\&+&\sum_{j=1}^{l}(\t^j_m)^{\frac{2-n}{n}}
\eta_r(\exp^{-1}_{x_m^j}(x))\nu_j((\t_m^j)^{-1}\exp^{-1}_{x_m^j}(x))+\W_m,\\&& \nonumber
\text{ with } \W_m\to 0 \text{ in }H^1_2(M),
\end{eqnarray}
and
\begin{eqnarray}\label{3.31}
\beta=J_h(u)+\sum_{i=1}^k
J_\infty(v_i)+\sum_{j=1}^l J(\nu_j)+o(1).
\end{eqnarray}
First, let  $\beta $ be such that
\begin{equation*}
0<\beta<D^*=\frac{(1-h(p)K^2(n,2,-2))^{\frac{n-1}{2}}}{nk(n,2)^n},
\end{equation*}
and suppose that $u\equiv0$. If there exists $i, 1\le i\le  k$ such that
$v_i\neq0$ and \eqref{3.30}, \eqref{3.31} hold, then by \eqref{0.4} we get
\begin{equation}\label{3.32}
\beta= J_\infty(v_i)\ge D^*.
\end{equation}
Thus, $\forall 1\le i\le k, v_i =0$. In the same way, if there exists $\nu_j$ such that \eqref{3.30} and  \eqref{3.31}, then
\begin{equation*}
    \beta= J(\nu_j)\ge d^*>D^*.
\end{equation*}
which also makes  a contradiction. Hence, $u\neq0$ which implies that $J_h(u)>0$ and thus, for the same reasons as above, the $u_m$ converges, up to a subsequence, to $u$ in $H^2_1(M)$.\\
For the second part of the corollary, let $\beta$ be such that $D^*<\beta<2D^*$, and suppose that $u\equiv0$. First, note that by the proof of theorem \ref{thm3.6}, $u\equiv0$ is nothing but the weak limit in $H^2_1(M)$ of $u_m$. Since $ \beta> D^*>0$, the sequence $u_m$ cannot converge strongly in $H_1^2(M)$ to $0$. Then, it follows from lemmas \ref{lem3.3} and \ref{lem3.4}  that there exists $v$ of \eqref{0.3} such that if $v\neq0$, the sequence
\begin{equation*}
    w_m=u_m-(R_m)^{\frac{2-n}{n}}
\eta_r(\exp^{-1}_p(x))v((R_m)^{-1}\exp^{-1}_p(x))
\end{equation*}
is a Palais-Smale sequence of $J_h$ that converges weakly to $0$ and
\begin{equation}\label{3.33}
    J_h(w_m)=\beta-J_\infty(v)+o(1).
\end{equation}
By \eqref{0.4},  since $D^*<\beta<2D^*$, either $v$ changes sing or not, $\beta\neq J_\infty(v)$. Hence,
\begin{equation*}
   0<\beta-J_\infty(v)<D^*
\end{equation*}
which implies by the first part of the corollary that the sequence $w_m$ converges  strongly to non zero function $w$ such that $DJ_h(w)=0$ which is impossible since, if it is the case, by \eqref{3.33} and \eqref{3.29}, it follows that
\begin{equation*}
\beta=J_h(w)+J_\infty(v)\ge\frac{\mu^{\frac{n}{2}}}{n}+D^*\ge2D^*,
\end{equation*} which is against \eqref{3.30'}. Hence, $v=0$.\\
On the other hand, by lemma \ref{lem3.5} there exists $\nu\neq0$ such that the sequence
\begin{eqnarray*}
    w_m&=&u_m-(\t_m)^{\frac{2-n}{n}}
\eta_r(\exp^{-1}_{x_m^j}(x))\nu((\t_m)^{-1}\exp^{-1}_{x_m}(x))
\end{eqnarray*}
is a Palais-Smale sequence that converges weakly to $0$ in $H^2_1(M)$ and
\begin{equation*}
   J_h(w_m) =\beta-J(\nu)+o(1).
\end{equation*}
Since $\nu$ is solution of \eqref{0.2}, then $J(\nu)\ge d^*=\frac{1}{nK^{n}(n,2)}$ which, under \eqref{3.28}, implies that
\begin{equation*}
J_h(w_m) =\beta-J(\nu)<0,
\end{equation*}
which also impossible since $w_n$ is a Palais-Smale sequence of $J_h$. Hence, the function $u$ under  conditions \eqref{3.28},\eqref{3.29} and \eqref{3.30} cannot be identically zero.\\
Now, suppose that there exists a solution $v_i\neq0$ of \eqref{0.3} such that \eqref{3.30} and \eqref{3.31}  hold,
 then
\begin{equation*}
    \beta=J_h(u)+J_\infty(v_i)\ge\frac{\mu^{\frac{n}{2}}}{n}+D^*\ge 2D^*.
\end{equation*}
Thus, $v_i\equiv0,\forall i,1\le i\le k$. Similarly, we obtain $\nu_j\equiv0,\forall j,1\le j\le l$. Hence, the sequence $u_m$ converges, up to a subsequence, strongly to $u\neq0$ in $H^2_1(M)$ with $DJ_h(u)=0$.
\end{proof}
\section{Existence of solutions}
As a consequence of  corollary \eqref{coro1}, we obtain the following existence result. We construct Palais-Smale sequences of the functional $ J_h$ at levels $\beta$ confined between the values given in corollary \ref{coro1}. In this way, we prove the following theorem
\begin{Theorem} \label{thm4.1}Let $M$ be a compact Riemannian manifold of dimension $n>\frac{2}{a}+2$, where $a$ is defined by \eqref{0.3}. Let $h$ be a smooth function  on $M$. Under the following conditions:
\begin{eqnarray}
   &&h(p)>0, 1-h(p)K^{2}(n,2,-2)>0\\
  && \label{4.36''} (A(n,a)+h(p))Scal_g(p)-\Delta h(p)<0 ,
\end{eqnarray} \text{ where } A(n,a) \text{ is defined by } \eqref{4.41'},\\
 there exists a non zero weak solution $u$ of \eqref{0.1} such that $0<J_h(u)<D^*$, where $D^*$ is defined by \eqref{0.4}.\\
 Moreover, under the following conditions:
 \begin{eqnarray}
   &&\label{4.35'} h(p)>0, 0<(1-h(p)K^{2}(n,2,-2))^{\frac{n-1}{2}}<\frac{1}{2} \\
   && \label{4.36'}(A(n,a)+h(p))Scal_g(p)-\Delta h(p)>0,\\
   &&\mu^{\frac{n}{2}}>nD^*, \text{where }\mu \text{ is defined by} \eqref{0.4}
\end{eqnarray}
there exists a non zero weak solution $u$ of \eqref{0.1} such that $D^*<J_h(u)<2D^*$.
\end{Theorem}
For the proof of the above theorem, we introduce the Nehari manifold for the functional $J_h$
\begin{equation*}
    \NN=\{u\in H^2_1(M)\setminus\{0\}, DJ_h(u).u=0\}.
\end{equation*}
Note that for each $u\in u\in H^2_1(M)\setminus\{0\}$, the function
\begin{equation*}
\Phi(u)=(\frac{\int_M(|\nabla u|^2-\frac{h}{r^2}u^2)dv_g}{\int_M|u|^{2^*}dv_g})^{\frac{n-2}{4}}u
\end{equation*}
 belongs to $\NN$ and $J_h(\Phi(u))=\max_{t>0}J(tu)$.\\
 Let $ 0\le\eta_\delta\le1$ be a cut-off function  on $M$ such that, $\eta_\delta(x)=1, x\in B(p,\delta), \eta_\delta(x)=0, x\in M\setminus\in B(p,2\delta)$ and $|\nabla\eta_\delta|\le C$, for some constant $C>0$.\\
 Put $\rho_p(x)=r$ and for $0<\delta<\frac{\delta_g}{2}$, consider  on $M$ the function
 \begin{equation*}
    \phi_\varepsilon(x)= C(n,a)
    \eta_\delta\left(\frac{\varepsilon^{a}r^{a-1}}{\varepsilon^{2a}+r^{2a}}\right)^{\frac{n}{2}-1},
\end{equation*}
where $C(n,a)= (a^2n(n-2))^{\frac{n-2}{4}}$ and  $a=\sqrt{1-h(p)K(n,2,-2)^2}$.\\
We begin by proving the following lemma:
\begin{Lemma}\label{lem4.3} There exists a constant $A(n,a)$ such that if
 \begin{enumerate}
   \item $n=dim(M)>2+\frac{2}{a},$ and
   \item $(A(n,a)+h(p))Scal_g(p)-6\frac{\Delta h(p)}{n}<0,$
 \end{enumerate}
 there holds
\begin{equation}\label{4.34}
 J_h(\Phi(\phi_\varepsilon)<D^*.
\end{equation}
\end{Lemma}
 \begin{proof}
 Consider a geodisic  normal coordinate system around $p$. In this system, the function $G(r)= \frac{1}{w_{n-1}}\int_{S^{n-1}}\sqrt{|g(x)}d\sigma$ writes (see for example the book \cite{Hebey}, page 283.)
 \begin{equation*}
G(r)=1-\frac{1}{6}Scal_g(p)r^2+o(r^2)
 \end{equation*}
 Define for $2a\beta-1>\alpha>0$
 \begin{equation*}
   I^\alpha_{\beta}= \int_0^{\infty}\frac{r^{\alpha}}{(1+ r^{2a})^\beta}dr.
\end{equation*}
Then, one can easily have for $\alpha>2a-1$ and $2a\beta>\alpha+1$
\begin{eqnarray}
 \label{4.30} I^\alpha_{\beta}&=&\frac{\alpha-2a+1}{2a\beta-(\alpha+1)}I_{\beta}^{\alpha-2a}  \\
 \label{4.31} I^{\alpha-2a}_{\beta-1}&=& \frac{2a(\beta-1)}{2a\beta-(\alpha+1)}I_{\beta}^{\alpha-2a}
\end{eqnarray}
 We have
 \begin{eqnarray*}
 &&\int_M|\nabla\phi_\varepsilon|^2dv_g\\&=&
C(n,a)^2 (\frac{n-2}{2})^2w_{n-1}\varepsilon^{a(n-2)}[\int_0^{\delta}\frac{r^{a(n-2)-1}((a-1)\varepsilon^{2a}-(1+a)r^{2a})^2}
 {(\varepsilon^{2a}+r^{2a})^n}G(r)dr\\&+& C(n,a)^2 (\frac{n-2}{2})^2w_{n-1}\int_\delta^{2\delta}\frac{r^{a(n-2)-1}((a-1)\varepsilon^{2a}-(1+a)r^{2a})^2}
 {(\varepsilon^{2a}+r^{2a})^n}\eta_\delta G(r)dr\\&+&(n-2)w_{n-1}C(n,a)\varepsilon^{\frac{a(2-n)}{2}}\int_\delta^{2\delta}\frac{r^{\frac{a(n-2)+n-2}{2}}( (a-1)\varepsilon^{2a}-(1+a)r^{2a})}{(\varepsilon^{2a}+r^{2a})^{\frac{n}{2}}}(\nabla r.\nabla\eta_\delta)G(r)dr \\&+& w_{n-1}C(n,a)^2\int_\delta^{2\delta}\frac{r^{(a-1)(n-2)+n-1}}
 {(\varepsilon^{2a}+r^{2a})^{n-2}}|\nabla\eta_\delta|^2G(r)dr]
 \end{eqnarray*}
 The change of variable $r=\varepsilon t$ gives
  \begin{eqnarray*}
 &&\int_M|\nabla\phi_\varepsilon|^2dv_g\\&=&w_{n-1}[
 C(n,a)^2(\frac{n-2}{2})^2\int_0^{\frac{\delta}{\varepsilon}}\frac{t^{a(n-2)-1}((a-1)-(1+a)t^{2a})^2}
 {(1+t^{2a})^n}G(\varepsilon t)dt\\&+& C(n,a)^2 (\frac{n-2}{2})^2\int_{\frac{\delta}{\varepsilon}}^{\frac{2\delta}{\varepsilon}}\frac{t^{a(n-2)-1}((a-1)-(1+a)t^{2a})^2}
 {(1+t^{2a})^n}\eta_\delta G(\varepsilon t)dt\\&+&(n-2)C(n,a)\varepsilon^{\frac{n+2}{2}}  \int_{\frac{\delta}{\varepsilon}}^{\frac{2\delta}{\varepsilon}}\frac{t^{\frac{a(n-2)+n-2}{2}}( (a-1)-(1+a)t^{2a})}{(1+t^{2a})^{\frac{n}{2}}}(\nabla t.\nabla\eta_\delta)G(\varepsilon t)dt\\&+& \varepsilon^{2} C(n,a)^2 \int_{\frac{\delta}{\varepsilon}}^{\frac{2\delta}{\varepsilon}}\frac{t^{(a-1)(n-2)+n-1}}
 {(1+t^{2a})^{n-2}}|\nabla\eta_\delta|^2G(\varepsilon t)dt
 \end{eqnarray*}
  The function $G(r)$ is bounded in $B(p,2\delta)$, then
  \begin{eqnarray*}
    &&\int_{\frac{\delta}{\varepsilon}}^{\frac{2\delta}{\varepsilon}}\frac{t^{a(n-2)-1}((a-1)-(1+a)t^{2a})^2}
 {(1+t^{2a})^n}\eta_\delta G(\varepsilon t)dt\\&&\le C \int_{\frac{\delta}{\varepsilon}}^{\frac{2\delta}{\varepsilon}}\frac{t^{a(n-2)-1}((a-1)-(1+a)t^{2a})^2}
 {(1+t^{2a})^n}dt\to0_{\varepsilon\to0},
  \end{eqnarray*}
  together with
 \begin{eqnarray*}
  &&\varepsilon^{\frac{a(n-2)+n+2}{2}}|\int_{\frac{\delta}{\varepsilon}}^{\frac{2\delta}{\varepsilon}}\frac{r^{\frac{a(n-2)+n-2}{2}}( (a-1)-(1+a)t^{2a})}{(1+t^{2a})^{\frac{n}{2}}}(\nabla t.\nabla\eta)G(\varepsilon t)dt|\\&&\le C\varepsilon^{\frac{a(n-2)}{2}}
 \end{eqnarray*}
and
\begin{equation*}
\varepsilon^{2}|\int_{\frac{\delta}{\varepsilon}}^{\frac{2\delta}{\varepsilon}}\frac{t^{(a-1)(n-2)+n-1}}
 {(1+t^{2a})^{n-2}}|\nabla\eta|^2G(\varepsilon t)dt|\le C\varepsilon^{a(n-2)+1}.
\end{equation*}
Observe that for $n>2+\frac{2}{a}$,
\begin{equation*}
   \lim_{\varepsilon\to0}\int_{\frac{\delta}{\varepsilon}}^{\infty}\frac{t^{a(n-2)+1}((a-1)-(1+a)t^{2a})^2}
 {(1+t^{2a})^n}dt=0
\end{equation*}
and
\begin{equation*}
   \lim_{\varepsilon\to0}\int_{\frac{\delta}{\varepsilon}}^{\infty}\frac{t^{a(n-2)-1}((a-1)-(1+a)t^{2a})^2}
 {(1+t^{2a})^n}dt=0.
\end{equation*}
Hence, we get
\begin{eqnarray*}
   &&\int_M|\nabla\phi_\varepsilon|^2dv_g\\&=& C(n,a)^2(\frac{n-2}{2})^2w_{n-1}[\int_0^{\infty}\frac{t^{a(n-2)-1}((a-1)-(1+a)t^{2a})^2}
 {(1+t^{2a})^n}dt\\&-&\frac{1}{6}Scal_g(p)\varepsilon^2\int_0^{\infty}\frac{t^{a(n-2)+1}((a-1)-(1+a)t^{2a})^2}
 {(1+t^{2a})^n}dt+\alpha_1(\varepsilon)].
\end{eqnarray*}
with $\lim_{\varepsilon\to0}\alpha_1(\varepsilon)=0$.\\
Now, by using \eqref{4.30}, we get
\begin{eqnarray*}
&&\int_0^{\infty}\frac{t^{a(n-2)+1}((a-1)-(1+a)t^{2a})^2}
 {(1+t^{2a})^n}dt\\&=&(a-1)^2I_n^{a(n-2)+1}-2(a^2-1)I_n^{an+1}+(1+a)^2I_n^{a(n+2)+1}
 \\&=&[(a-1)^2+2(1-a^2)\frac{a(n-2)+2}{an-2}\\&+&(1+a)^2\frac{(an+2)(a(n-2)+2)}{(an-2)(a(n-2)-2)}]I_n^{a(n-2)+1}
\end{eqnarray*}
Finally, by observing  that
\begin{equation*}
C(n,a)^2(\frac{n-2}{2})^2w_{n-1}\int_0^{\infty}\frac{t^{a(n-2)+1}((a-1)-(1+a)t^{2a})^2}
 {(1+t^{2a})^n}dt=\int_{\R^n}|\nabla U|^2dx,
\end{equation*}
where
 \begin{equation*}
    U(x)=C(n,a)\left(\frac{|x|^{a-1}}{(1+|x|^{2a})}\right)^{\frac{n-2}{2}},x\in \R^n.
 \end{equation*}
 we get
\begin{eqnarray}\label{4.31}
&&\int_M|\nabla\phi_\varepsilon|^2dv_g\\ \nonumber &=&\int_{\R^n}|\nabla U|^2dx -Scal_g(p)C_1(n,a)
I_n^{a(n-2)+1}\varepsilon^2+o(\varepsilon^2)+\alpha_1(\varepsilon).
\end{eqnarray}
with
\begin{eqnarray*}
C_1(n,a)&=&\frac{1}{6}C(n,a)^2(\frac{n-2}{2})^2w_{n-1}[(a-1)^2+2(1-a)\frac{a(n-2)+2}{an-2}\\&+&(1+a)^2\frac{(an+2)(a(n-2)+2)}{(an-2)(a(n-2)-2)}].
\end{eqnarray*}
Similarly we develop the term $\int_M\frac{h(x)}{r^2}\phi_\varepsilon^2dv_g $. First, by choosing $\delta$ small we can write for $x\in B(p,\delta)$
\begin{equation*}
    h(x)=h(p)+(\nabla_ih)(p)x_i+(\nabla_{i,j}h)(p)x_ix_j+o(r^2)
\end{equation*}
 Using the fact that $\int_{S^{n-1}}x_id\sigma=0 $ and $\int_{S^{n-1}}x_ix_jd\sigma=0 $ for $i\neq j$, we obtain
\begin{eqnarray*}
&&\int_M\frac{h(x)}{r^2}\phi_\varepsilon^2dv_g\\&=&h(p)\int_{\R^N} \frac{U^2}{|x|^2}dx\\&-&w_{n-1}C(n,a)^2Scal_g(p)[\frac{Scal_g(p)h(p)}{6}-\frac{\Delta h(p)}{n}]I^{a(n-2)+1}_{n-2}\varepsilon^2+o(\varepsilon^2)+\alpha_2(\varepsilon),
\end{eqnarray*}
with $\lim_{\varepsilon\to0}\alpha_2(\varepsilon)=0$.\\
using , \eqref{4.31} we get
\begin{eqnarray}\label{4.32}
&&\nonumber \int_M\frac{h(x)}{r^2}\phi_\varepsilon^2dv_g\\ &=& h(p)\int_{\R^N} \frac{U^2}{|x|^2}dx-[Scal_g(p)h(p)-\frac{\Delta h(p)}{n}]C_2(n,a)I^{a(n-2)+1}_{n}\varepsilon^2\\ \nonumber&+&o(\varepsilon^2)+\alpha_2(\varepsilon)],
\end{eqnarray}
with
\begin{equation*}
C_2(n,a)=C(n,a)^2\frac{4a^2w_{n-1}(n-2)(n-1)}{(a(n-2)-2)(an-2)}.
\end{equation*}
As to the term $\int_M|\phi_\varepsilon|^{2^*}dv_g$, it develops as
\begin{equation*}
\int_M|\phi_\varepsilon|^{2^*}dv_g= \int_{\R^n}|U|^{2^*}dx-C(n,a)^{2^*}\frac{1}{6}w_{n-1}Scal_g(p)I_n^{an+1}\varepsilon^2+
o(\varepsilon^2)+\alpha_3(\varepsilon),
\end{equation*}
with $\lim_{\varepsilon\to0}\alpha_3(\varepsilon)=0$. Again, we use \eqref{4.30}, we obtain
\begin{eqnarray}\label{4.33}
&&\int_M|\phi_\varepsilon|^{2^*}dv_g\\ \nonumber&=& \int_{\R^n}|U|^{2^*}dx-Scal_g(p)C_3(n,a)I_n^{a(n-2)+1}\varepsilon^2+o(\varepsilon^2)
\\&+& \nonumber\alpha_3(\varepsilon),
\end{eqnarray}
with
\begin{equation*}
C_3(n,a)=C(n,a)^{2^*}\frac{w_{n-1}(a(n-2)+2)n}{6(an-2)}.
\end{equation*}
Using the fact that
\begin{equation*}
    \frac{\int_{\R^n}(|\nabla U|^2-h(p)\frac{U^2}{|x|^2})dx}{(\int_{\R^n}|U|^{2^*}dx)^{\frac{2}{2^*}}}=
    \frac{(1-h(p)K(n,2,-2)^2)^{\frac{n-1}{n}}}{K(n,2)^2}=(nD^*)^{\frac{2}{n}},
\end{equation*}
the expansions \eqref{4.31}, \eqref{4.32} and \eqref{4.33} give
\begin{eqnarray*}
    &&\frac{ (\int_M|\nabla\phi_\varepsilon|^2-\frac{h}{r^2}\phi_\varepsilon^2)dv_g}
    {(\int_M|\phi_\varepsilon|^{2^*}dv_g)^{\frac{2}{2^*}}}\\&=&(nD^*)^{\frac{2}{n}}(1+
    \frac{1}{\int_{\R^n}|U|^{2^*}dx}[Scal_g(p)(\frac{n-2}{n}C_3(n,a)-C_1(n,a))\\&+&(\frac{Scal_g(p)h(p)}{6}-\frac{\Delta h(p)}{n})C_2(n,a))]I_n^{a(n-2)+1}\varepsilon^2+o(\varepsilon^2)+\alpha(\varepsilon).
\end{eqnarray*}
with $\lim_{\varepsilon\to0}\alpha(\varepsilon)=0$.\\
Now, writing
\begin{equation*}
    nJ(\Phi(\phi_\varepsilon))=\left(\frac{ (\int_M|\nabla\phi_\varepsilon|^2-\frac{h}{r^2}\phi_\varepsilon^2)dv_g}
    {(\int_M|\phi_\varepsilon|^{2^*}dv_g)^{\frac{2}{2^*}}}\right)^{\frac{n}{2}}
\end{equation*}
 we obtain
 \begin{eqnarray*}
     &&J_h(\Phi(\phi_\varepsilon))\\&=&D^*(1+
    \frac{n}{2(\int_{\R^n}|U|^{2^*}dx)^{\frac{n}{2}}}[Scal_g(p)((\frac{n-2}{n}C_3(n,a)-C_1(n,a))\\&+&(\frac{Scal_g(p)h(p)}{6}-\frac{\Delta h(p)}{n})C_2(n,a))]I_n^{a(n-2)+1}\varepsilon^2+o(\varepsilon^2)+\alpha(\varepsilon).
 \end{eqnarray*}
Finally, take
\begin{equation}\label{4.41'}
 A(n,a)= \frac{6(\frac{n-2}{n}C_3(n,a)-C_1(n,a))}{C_2(n,a)},
\end{equation}
and
\begin{equation}\label{4.42}
 B(n,a)=\frac{n}{12C_2(n,a)(\int_{\R^n}|U|^{2^*}dx)^{\frac{n}{2}}}
 \end{equation}
 we get
 \begin{eqnarray}\label{4.45}
 &&\nonumber J_h(\Phi(\phi_\varepsilon)\\&=&D^*\left[1+
    B(n,a)\left((A(n,a)+h(p))Scal_g(p)-6\frac{\Delta h(p)}{n}\right)\varepsilon^2\right]\\ \nonumber&+&o(\varepsilon^2)+\alpha(\varepsilon)\nonumber.
 \end{eqnarray}
 Hence, if
 \begin{equation*}
(B(n,a)+h(p))Scal_g(p)-6\frac{\Delta h(p)}{n}<0,
 \end{equation*}
 we get \eqref{4.34}.
\end{proof}
\begin{proof}[Proof of theorem \ref{thm4.1}] Since the functional $J_h$ is bounded from below on the Nehari
manifold $\NN$, the variational principle of Ekland gives a Palais-Smale sequence  $u_m\in \NN$ of $J_h$ at the level $\beta=\inf_{u\in\NN,u\neq0}J_h(u)>0$. By definition of the manifold $\NN$,  $u_m$ is still a Palais-Smale sequence of $J_h$ on $H^2_1(M)$. Under condition \eqref{4.36''} of the theorem, lemma \ref{lem4.3} implies that $\beta\le J_h(\Phi(\phi_\varepsilon))<D^*$. Therefore, $u_m$ converges, by corollary \ref{coro1}, strongly in $H^2_1(M)$ to a nontrivial solution $u$ of \eqref{0.1} which then satisfies $ 0<J_h(u)<D^*$.\\
For the second part of the theorem, Since $\mu^{\frac{n}{2}}>nD^*$, then $\inf_{u\in\NN} J_h(u)>D^*$. On the other hand, the expansion \eqref{4.45} together with condition \eqref{4.35'} of the theorem give that $ \inf_{u\in\NN} J_h(u)<2D*$. Now, again by the Ekland variational principle there exists a Palais-Smale sequence at level $\beta=\inf_{u\in\NN} J_h(u)$, which by corollary \ref{coro1} converges, up to a subsequence to a weak solution with $D^*<J_h(u)<2D^*$.
\end{proof}

\end{document}